\documentclass[11pt]{amsart}
\usepackage{graphicx}
\usepackage{amsmath} 
\usepackage{amsthm,amsfonts,amssymb,mathrsfs,amscd,amstext,amsbsy} 
\usepackage{epic,eepic} 
\usepackage{yfonts}
\usepackage{paralist,enumerate}
\usepackage[all]{xy}
\usepackage{hyperref}
\hypersetup{colorlinks}

\newtheorem{theorem}{Theorem}[section]
\newtheorem{cor}[theorem]{Corollary}
\newtheorem{lemma}[theorem]{Lemma}
\newtheorem{conj}[theorem]{Conjecture}

\newtheorem{theo}[theorem]{Theorem}
\newtheorem{lem}[theorem]{Lemma}
\newtheorem{pro}[theorem]{Proposition}
\newtheorem{rem}[theorem]{Remark}
\newtheorem{exa}[theorem]{Example}
\newtheorem{que}[theorem]{Question}

\newtheorem*{Definition*}{Definition}

\def\qed{\hfill \ifhmode\unskip\nobreak\fi\quad\ifmmode\Box\else$\Box$\fi\\ }

\begin{document}

\title[Kosniowski's conjecture and weights]{Kosniowski's conjecture and weights}
\author{Donghoon Jang}
\thanks{Donghoon Jang is supported by Basic Science Research Program through the National Research Foundation of Korea(NRF) funded by the Ministry of Education(2018R1D1A1B07049511).}
\address{Department of Mathematics, Pusan National University, Pusan, Korea}
\email{donghoonjang@pusan.ac.kr}

\begin{abstract}
The conjecture of Kosniowski asserts that if the circle acts on a compact unitary manifold $M$ with a non-empty fixed point set and $M$ does not bound a unitary manifold equivariantly, then the dimension of the manifold is bounded above by a linear function on the number of fixed points. We confirm the conjecture for almost complex manifolds under an assumption on weights. For instance, we confirm the conjecture if there is one type $\pm a$ of weights, or there are two types $\pm a$, $\pm b$ of weights.
\end{abstract}

\maketitle

\tableofcontents

\section{Introduction and statements of main results}

$\indent$

The main purpose of this paper is to consider Kosniowski's conjecture for almost complex manifolds, in terms of the number of types of weights. The conjecture of Kosniowski asserts that if the circle acts on a compact unitary manifold $M$ with a non-empty fixed point set and $M$ does not bound a unitary manifold equivariantly, then there is a relationship between the dimension of the manifold and the number of fixed points. Kosniowski conjectures that the dimension of $M$ is bounded above by a (linear) function $f(k)$ on the number $k$ of fixed points. He conjectures further that $f(k)=4k$. 

\begin{conj} [Kosniowski] \label{co1} \cite{K2} Let the circle act on a compact unitary manifold $M$ with a fixed point. Suppose that $M$ does not bound a unitary manifold equivariantly. Then $\dim M < 4k$, where $k$ is the number of fixed points. Alternatively, there exist at least $\lfloor \frac{\dim M}{4} \rfloor+1$ fixed points. \end{conj}

Here, $\lfloor u \rfloor$ is the greatest integer less than or equal to $u$ for each real number $u$. In \cite{K2}, Kosniowski says that he has some clue for Conjecture \ref{co1}, but the clue has not been spelled out, to the author's knowledge. There are various results which are related to Kosniowski's conjecture; see \cite{CKP}, \cite{GPS}, \cite{LL}, \cite{LL2}, \cite{LT}, \cite{L2}, \cite{PT}, etc. However, most of them do not confirm Conjecture \ref{co1} in an arbitrary dimension under their additional assumptions. Moreover, most results show different bounds than the original conjecture. Many bounds of them involve characteristic numbers; Chern numbers for (stable) almost complex manifolds and Pontryagin numbers for oriented manifolds. Many of the results involving characteristic numbers share the same method of utilizing a Vandermonde matrix which is first used in \cite{PT}.

In this paper, we attempt to approach Conjecture \ref{co1} for almost complex manifolds, in terms of isotropy weights of the action. The almost complex manifold version of Kosniowski's conjecture can be stated as follows.

\begin{conj} [Kosniowski] \label{co2} \cite{K2} Let the circle act on a compact almost complex manifold $M$ with a fixed point. Then $\dim M < 4k$, where $k$ is the number of fixed points. Alternatively, there exist at least $\lfloor \frac{\dim M}{4} \rfloor+1$ fixed points. \end{conj}

Let the circle act on an almost complex manifold $(M,J)$. Throughout this paper, assume that the action preserves the almost complex structure $J$. Let $p$ be an isolated fixed point of the action. Then the local action of $S^1$ near $p$ is described by
\begin{center}
$g \cdot (z_1,\cdots,z_n)=(g^{w_{p1}}z_1,\cdots,g^{w_{pn}}z_n)$
\end{center}
for some non-zero integers $w_{pi}$, $1 \leq i \leq n$, where $g \in S^1 \subset \mathbb{C}$, $z_i$ are complex coordinates, and $\dim M=2n$. The integers $w_{pi}$ are called \textbf{weights} at $p$. 

For an almost complex manifold $M$, the Hirzebruch $\chi_y$-genus $\chi_y(M)$ of $M$ is the genus belonging to the power series $\frac{x(1+ye^{-x(1+y)})}{1+e^{-x(1+y)}}$. Let $\chi_y(M)=\sum_{i=0}^n \chi^i(M) y^i$, where $\dim M=2n$. Let the circle act on a compact almost complex manifold $M$ with a discrete fixed point set. Then the Hirzebruch $\chi_y$-genus $\chi_y(M)=\sum_{i=0}^n \chi^i(M) y^i$ of $M$ satisfies $\chi^i(M)=(-1)^i N^i$, where $N^i$ is the number of fixed points which have exactly $i$ negative weights; see \cite{L1}.

In this paper, we propose questions which are related to Kosniowski's conjecture. First, we raise a question if the Hirzebruch $\chi_y$-genus is crowded.

\begin{que} \label{q12} Let the circle act on a compact almost complex manifold $M$ with a discrete fixed point set. Suppose that $\chi^{j_1}(M) \neq 0$ and $\chi^{j_2}(M) \neq 0$ for some non-negative integers $j_1$ and $j_2$. Then for each $j$ between $j_1$ and $j_2$, is $\chi^j(M) \neq 0$?

Alternatively, suppose that there are a fixed point with $j_1$ negative weights and a fixed point with $j_2$ negative weights. Then for each $j$ between $j_1$ and $j_2$, does there exist a fixed point with $j$ negative weights? \end{que}

Note that the answer to Question \ref{q12} is yes in dimension up to 6; see Proposition \ref{p24}.

Question \ref{q12} might be the original motivation for Conjecture \ref{co1}. In fact, in \cite{K1} Kosniowski proves that (a) if a holomorphic vector field (circle action) on a compact complex manifold $M$ has simple isolated zeroes (isolated fixed points), then there exists $i$ such that both of $\chi^i(M)$ and $\chi^{i+1}(M)$ are non-zero. In other words, there exist a non-negative integer $i$ and two simple isolated zeroes (fixed points) $p$ and $q$ such that $p$ has $i$ negative eigenvalues (weights) and $q$ has $i+1$ negative eigenvalues (weights). With this, in \cite{K1} Kosniowski proves that if there are two simple isolated zeroes, then $\dim M=2$ or 6. 

Moreover, (b) the Hirzebruch $\chi_y$-genus is symmetric, that is, $\chi^i(M)=(-1)^{n}\chi^{n-i}(M)$ and hence $N^i=N^{n-i}$ for each $i$, where $\dim M=2n$. In other words, the number of fixed points with $i$ negative weights is equal to the number of fixed points with $n-i$ negative weights for each $i$. Then (a) and (b) imply for instance that if there are three fixed points, then $n$ must be even, and there must be one fixed point with $\frac{n}{2}-1$ negative weights, one fixed point with $\frac{n}{2}$ negative weights, and one fixed point with $\frac{n}{2}+1$ negative weights. From these observations, it is plausible to ask Question \ref{q12}.

More precisely, the following question might be Kosniowski's motivation for his conjecture. For a real number $u$, let $\lceil u \rceil$ be the least integer greater than or equal to $u$.

\begin{que} \label{q13} Let the circle act on a $2n$-dimensional compact almost complex manifold $M$ with a non-empty discrete fixed point set. Then for any $i$ such that $\lfloor \frac{n}{3} \rfloor \leq i \leq \lceil \frac{2n}{3} \rceil$, is $\chi^i(M) \neq 0$? In other words, for any $i$ such that $\lfloor \frac{n}{3} \rfloor \leq i \leq \lceil \frac{2n}{3} \rceil$, does there exist a fixed point with $i$ negative weights? \end{que}

All the known examples of almost complex $S^1$-manifolds with discrete fixed point sets satisfy the properties in Question \ref{q12} and Question \ref{q13}. Therefore, it is possible that combined with the above observations (a) and (b) and the example of $S^6$ and all the known examples, Kosniowski might make Conjecture \ref{co1} based on Question \ref{q12} and Question \ref{q13}.

Suppose that the answer to Question \ref{q13} is affirmative. Then this means there are at least $\lceil \frac{2n}{3} \rceil-\lfloor \frac{n}{3} \rfloor+1$ fixed points. And this means that the dimension of a manifold is bounded above by a linear function $f(k)$ of the number $k$ of fixed points. Then from the 6-dimensional example of a rotation on $S^6$ with 2 fixed points and the 12-dimensional example of $S^6 \times S^6$ with 4 fixed points, $f(k)=4k$ is the sharpest function. This might be how Kosniowski comes up with Conjecture \ref{co1}.

Let the circle act on a compact almost complex manifold $M$ with a fixed point. On each dimension, the following table shows the lower bound $\lfloor \frac{\dim M}{4} \rfloor+1$ for the number of fixed points by Kosniowski, the actual possible minimum (if we know), the minimum number of fixed points by a known manifold, and the known manifold.

\begin{center}
\begin{tabular}{|p{2cm}|p{2cm}|p{2cm}|p{2.7cm}|p{2cm}|}
\hline
\multicolumn{5}{|c|}{Dimension and minimal number of fixed points} \\
\hline
Dimension & Bound by Kosniowski & Actual minimum & Minimum by known manifold & The known manifold \\
\hline
2  & 1 & 2 & 2 & $S^2$ \\
4  & 2 & 3 & 3 & $\mathbb{CP}^2$ \\
6  & 2 & 2 & 2 & $S^6$ \\
8  & 3 & 4 & 4 & $S^2 \times S^6$\\
10 & 3 & ?($\geq$ 4) & 6 & $\mathbb{CP}^5$ or $\mathbb{CP}^2 \times S^6$ \\
12 & 4 & 4 & 4 & $S^6 \times S^6$ \\
14 & 4 & ?($\geq$ 4) & 8 & $\mathbb{CP}^{7}$ or $S^2 \times S^6 \times S^6$ \\
16 & 5 & ? & 9 & $\mathbb{CP}^{8}$ \\
18 & 5 & ? & 8 & $S^6 \times S^6 \times S^6$ \\
20 & 6 & ? & 11 & $\mathbb{CP}^{10}$ \\
22 & 6 & ? & 12 & $\mathbb{CP}^{11}$ \\
24 & 7 & ? & 13 & $\mathbb{CP}^{12}$ \\
$\vdots$ & $\vdots$ & $\vdots$ & $\vdots$ & $\vdots$ \\
$2n$ & $\lfloor\frac{n}{2}\rfloor+1$ & ? & $n+1$ & $\mathbb{CP}^n$ \\
$\vdots$ & $\vdots$ & $\vdots$ & $\vdots$ & $\vdots$ \\
\hline
\end{tabular}
\end{center}

For instance, in any positive dimension the actual minimum is strictly bigger than 1, since a circle action on a compact almost complex manifold can have exactly 1 fixed point only if the manifold is itself the point. On dimension 4, the actual minimum is 3 since two fixed points can only occur in dimension 2 and 6. As we see from the table, in low dimensions $S^6$ serves as a minimal model which provides the minimum number of fixed points among known manifolds, but fails to do so in dimensions beyond 18. In dimensions beyond 18, among known manifolds the complex projective spaces $\mathbb{CP}^n$ provide minimal numbers of fixed points. The classification results mentioned above appear in \cite{J3}.

\begin{theo} \cite{J3} \label{t12} Let the circle act on a compact almost complex manifold $M$. If there is one fixed point, then $M$ is a point. If there are two fixed points, then either $M$ is the 2-sphere or $\dim M=6$. If there are three fixed points, then $\dim M=4$. \end{theo}

In fact, the author proves that if $\dim M=6$ and there are two fixed points, then weights at the fixed points are the same as a rotation on $S^6$ given in Example \ref{e1} below, and if there are three fixed points, then $\dim M=4$ and weights at the fixed points are the same as a standard linear circle action on $\mathbb{CP}^2$ given in Example \ref{e3} below. Theorem \ref{t12} confirms Conjecture \ref{co2} if there are few fixed points, or equivalently, in low dimensions.

\begin{cor} \label{t13} Let the circle act on a compact almost complex manifold $M$ with a non-empty fixed point set. If $\dim M \leq 14$, then there are at least $\lfloor \frac{\dim M}{4} \rfloor+1$ fixed points. \end{cor}

The proof of the case of three fixed points in Theorem \ref{t12} adapts the proof for symplectic circle action in \cite{J1} and is however very complicated. In addition, the author classifies the case of four fixed points when the dimension of a manifold is at most six \cite{J4}, and this also involves many cases to consider. As the dimension of a manifold or the number of fixed points increases, the classification of a circle action becomes difficult more and more, and so does the Kosniowski's conjecture.

We discuss on the bound $f(k)=4k$ in Conjectures \ref{co1} and \ref{co2}. Among the known examples, the sharpest bound is achieved by the product of $S^6$'s and $\mathbb{CP}^n$'s with standard circle actions.

\begin{exa} \label{e1}
Regarding as $G_2/SU(3)$, the 6-sphere $S^6$ admits a circle action that has precisely two fixed points. The weights at the fixed points are $\{-a-b,a,b\}$ and $\{-a,-b,a+b\}$ for some positive integers $a$ and $b$.
\end{exa}

\begin{exa} \label{e3}
Let the circle act on $\mathbb{CP}^n$ by $g \cdot [z_0:\cdots:z_n]=[g^{a_0}z_0:\cdots:g^{a_n}z_n]$ for mutually distinct positive integers $a_0,\cdots,a_n$. The action has $n+1$ fixed points $[1:0:\cdots:0], \cdots, [0:\cdots:0:1]$. The weights at the fixed points are $\{a_1-a_0,\cdots,a_n-a_0\}, \cdots, \{a_0-a_n,\cdots,a_{n-1}-a_n\}$.
\end{exa}

Given a weight $+w$ which occurs at some fixed point, there exists a fixed point which has weight $-w$; see Proposition 2.11 of \cite{H} or Theorem 3.5 of \cite{L1}. Therefore, in terms of the number of types of weights, the first instance to approach Conjecture \ref{co2} is that all the weights at the fixed points are either $+w$ or $-w$ for some positive integer $w$. In this case, in \cite{J4} the author proves that the number of fixed points is equal to $l \cdot 2^n$ for some positive integer $l$, where $\dim M=2n$; see Lemma \ref{l22}. In particular, the number $l \cdot 2^n$ of fixed points satisfies $\lfloor \frac{\dim M}{4} \rfloor+1 \leq l \cdot 2^n$ in Conjecture \ref{co2}. Therefore, we have that

\begin{pro} \label{t15} Let the circle act on a compact almost complex manifold $M$ with a non-empty discrete fixed point set. If every weight is either $+w$ or $-w$ for some positive integer $w$, then there are $l \cdot 2^n$ fixed points for some positive integer $l$, where $\dim M=2n$. In particular, the number $l \cdot 2^n$ of fixed points satisfies $\lfloor \frac{\dim M}{4} \rfloor+1 \leq l \cdot 2^n$. \end{pro}

For a point $x$ in a manifold $M$ equipped with an $S^1$-action, the isotropy group $G_x$ of $S^1$ with respect to $x$ is the group of elements in $S^1$ which act trivially on $x$, i.e., $G_x=\{g \in S^1\, | \,g \cdot x = x\}$. If a fixed point $p$ has a weight $w$, then $p$ has isotropy group $\mathbb{Z}_w$. In other words, Proposition \ref{t15} says that Conjecture \ref{co2} holds if a manifold has three (two if $w=1$) isotropy subgroups of $S^1$, the trivial group $\{e\}$, $\mathbb{Z}_w$ ($\mathbb{Z}_w=\{e\}$ if $w=1$), and the whole group $S^1$.

One of the main goals of this paper is to prove that Conjecture \ref{co2} holds if there are two types of weights. In other words, we prove that Conjecture \ref{co2} holds if there are at most 4 isotropy subgroups of $S^1$, the trivial group $\{e\}$, $\mathbb{Z}_a$, $\mathbb{Z}_b$, and $S^1$ (Some of $\{e\}$, $\mathbb{Z}_a$, and $\mathbb{Z}_b$ may be equal).

\begin{theo} \label{t16} Let the circle act on a compact almost complex manifold $M$ with a non-empty fixed point set. Suppose that every weight is either $\pm a$ or $\pm b$ for some positive integers $a$ and $b$. Then there are at least $\lfloor \frac{\dim M}{4} \rfloor+1$ fixed points. \end{theo}

The following theorem confirms Conjecture \ref{co2} if a circle action has a small number of types of weights compared to the dimension of a manifold and the weights are pairwise relatively prime.

\begin{theo} \label{t32} Let the circle act on a compact almost complex manifold with a non-empty fixed point set. Suppose that every weight at any fixed point is equal to $\pm w_i$ for some positive integers $w_i$, $1 \leq i \leq l$, such that $1<w_i$ and $w_i$ are pairwise relatively prime. If $\displaystyle l < \frac{\dim M}{2[(\log_2 \dim M) -2]}$, then there exist at least $\lfloor \frac{\dim M}{4} \rfloor+1$ fixed points. \end{theo}

In particular, Conjecture \ref{co2} holds for three types of weights which are pairwise relatively prime and greater than 1.

\begin{theo} \label{t113} Let the circle act on a compact almost complex manifold $M$ with a non-empty fixed point set. Suppose that every weight is $\pm a_1$, $\pm a_2$, or $\pm a_3$ for some positive integers $a_i$ such that $1<a_i$ and $a_i$ are pairwise relatively prime. Then there exist at least $\lfloor \frac{\dim M}{4} \rfloor+1$ fixed points. \end{theo}

\begin{rem} The same conclusion for 4 types of weights as in Theorem \ref{t113} follows immediately from Theorem \ref{t32} except dimension 16. \end{rem}

\section{Preliminaries}

In this section, we review preliminary results to prove our results. In \cite{L1}, Li shows that an equivariant index of Dolbeault-type operators on a compact almost complex manifold equipped with a circle action having a discrete fixed point set is rigid. As a result, Li proves the following. As in the introduction, let $N^i$ be the number of fixed points which have exactly $i$ negative weights.

\begin{theorem} \label{t20} \cite{L1} Let the circle act on a $2n$-dimensional compact almost complex manifold $M$ with isolated fixed points. For each integer $i$ such that $0 \leq i \leq n$,
\begin{center}
$\displaystyle \chi^i(M)=\sum_{p \in M^{S^1}} \frac{\sigma_i (t^{w_{p1}}, \cdots, t^{w_{pn}})}{\prod_{j=1}^n (1-t^{w_{pj}})} = (-1)^i N^i = (-1)^i N^{n-i}$,
\end{center}
where $t$ is an indeterminate, $\sigma_i$ is the $i$-th elementary symmetric polynomial in $n$ variables, and $\chi_y(M)=\sum_{i=0}^n \chi^i(M) \cdot y^i$. \end{theorem}

Using Theorem \ref{t20}, the author and Tolman prove the following lemma, closely following Kosniowski's idea for a holomorphic vector field on a complex manifold with simple isolated zeroes \cite{K1}. Let circle act on a compact almost complex manifold with a discrete fixed point set. Let $p$ be a fixed point. Let $n_p$ be the number of negative weights at $p$. For each integer $w$, let $N_p(w)$ be the number of times the weight $w$ occurs at the fixed point $p$.

\begin{lem} \cite{JT}, \cite{J2} \label{l21} Let the circle act on a $2n$-dimensional compact almost complex manifold $M$ with a discrete fixed point set. Let $w$ be the smallest positive weight that occurs among all the weights at the fixed points. For each $i \in \{0,1,\cdots,n-1\}$, the number of times the weight $+w$ occurs at fixed points which have exactly $i$ negative weights, counted with multiplicity, is equal to the number of times the weight $-w$ occurs at fixed points which have exactly $i+1$ negative weights, counted with multiplicity. In other words, if $w$ is the smallest positive weight, then for each $0 \leq i \leq n-1$,
\begin{center}
$\displaystyle \sum_{p \in M^{S^1}, n_p=i} N_p(w)=\sum_{p \in M^{S^1}, n_p=i+1} N_p(-w)$.
\end{center} \end{lem}

Using Theorem \ref{t20}, the author classifies a circle action which has only one type $\pm w$ of weights. It also follows from Lemma \ref{l21}.

\begin{lem} \cite{J4} \label{l22} Let the circle act on a $2n$-dimensional compact almost complex manifold with a non-empty discrete fixed point set. Assume that all the weights are $\pm w$ for some positive integer $w$. Then the number of fixed points is $l \cdot 2^n$ for some positive integer $l$. Moreover, $N^i=l \cdot {n \choose i}$, where $N^i$ is the number of fixed points that have exactly $i$ negative weights. \end{lem}

An example of a manifold in Lemma \ref{l22} is a diagonal action on the product of $n$-copies of $S^2$'s, where on each $S^2$ the circle acts by rotation at speed $w$. By taking $l$ disjoint copies of such manifolds, the resulting manifold has $l \cdot 2^n$ fixed points. We end this section by showing that the answer to Question \ref{q12} is affirmative in dimension up to 6.

\begin{pro} \label{p24} Let the circle act on a compact almost complex manifold $M$ with a discrete fixed point set. Assume that $\dim M \leq 6$. If $\chi^{j_1}(M) \neq 0$ and $\chi^{j_2}(M) \neq 0$ for some non-negative integers $j_1$ and $j_2$, then for any $j$ between $j_1$ and $j_2$, $\chi^j(M) \neq 0$. \end{pro}

\begin{proof} Let $\dim M=2n$. By Lemma \ref{l21}, there exists $0 \leq i \leq n-1$ such that both of $\chi^i(M)$ and $\chi^{i+1}(M)$ are non-zero. This proves the proposition if $\dim M=2$. By Theorem \ref{t20}, $\chi^i(M)=(-1)^i N^i=(-1)^i N^{n-i}=(-1)^i (-1)^{n-i} \chi^{n-i}(M)=(-1)^n\chi^{n-i}(M)$ for any $0\leq i \leq n$. This proves the proposition if $\dim M=4$.

Suppose that $\dim M=6$. The proposition follows immediately except the following values for $(j_1,j_2)$; $(0,2)$, $(0,3)$, and $(1,3)$. Suppose that $(j_1,j_2)=(0,2)$. This means that $\chi^0(M)\neq 0$ and $\chi^2(M) \neq 0$. By Theorem \ref{t20}, we have that $\chi^3(M)=(-1)^3\chi^0(M) \neq 0$ and $\chi^1(M)=(-1)^3 \chi^2(M) \neq 0$ and hence the proposition follows. The case for $(j_1,j_2)=(1,3)$ is similar. Suppose that $(j_1,j_2)=(0,3)$. By Lemma \ref{l21}, there exists $0 \leq i \leq n-1$ such that both of $\chi^i(M)$ and $\chi^{i+1}(M)$ are non-zero. It follows that either $\chi^1(M) \neq 0$ or $\chi^2(M) \neq 0$. In either case, by Theorem \ref{t20}, we have that $\chi^2(M)=(-1)^3\chi^1(M)\neq 0$. \end{proof}

In particular, Proposition \ref{p24} says that in dimension 4, $\chi^i(M)\neq 0$ ($N^i>0$) for all $0 \leq i \leq 2$, and in dimension 6, either $\chi^1(M)=-\chi^2(M)\neq0$ and $\chi^0(M)=\chi^3(M)=0$ ($N^1=N^2>0$ and $N^0=N^3=0$), or $\chi^i(M) \neq 0$ ($N^i>0$) for all $0 \leq i \leq 3$.

\section{Proofs}

In this section, we prove our main results, Theorem \ref{t16}, Theorem \ref{t32}, and Theorem \ref{t113}. Recall that for a circle action on a compact almost complex manifold with a discrete fixed point set, at each fixed point $p$, $n_p$ denotes the number of negative weights at $p$ and $N_p(w)$ denotes the number of times the weight $w$ occurs at $p$. 

\begin{lemma} \label{l31} Let the circle act on a $2n$-dimensional compact almost complex manifold $M$ with a discrete fixed point set. Let $w$ be the smallest positive weight which occurs as a weight among the weights at all the fixed points. Suppose that the number of times the weights $+w$ and $-w$ occurs at each fixed point is bigger than or equal to $\frac{3 \dim M}{8}$, i.e., $N_p(w)+N_p(-w) \geq \frac{3 \dim M}{8}$ for any $p \in M^{S^1}$. Then for each $i$ such that $\lfloor \frac{n}{4} \rfloor \leq i \leq \lceil \frac{3n}{4} \rceil$, there exists a fixed point with $i$ negative weights. In particular, there are at least $\lfloor \frac{\dim M}{4} \rfloor+1$ fixed points. \end{lemma}

\begin{proof}
Pick any fixed point $p$. Then there are three possibilities:
\begin{enumerate}
\item $n_p \leq \lfloor \frac{n}{4} \rfloor$.
\item $\lfloor \frac{n}{4} \rfloor < n_p < \lceil \frac{3n}{4} \rceil$.
\item $\lceil \frac{3n}{4} \rceil \geq n_p$.
\end{enumerate}

We show that in any of the three cases above, the case (2) holds, i.e., there always exists a fixed point $q$ such that $\lfloor \frac{n}{4} \rfloor < n_q < \lceil \frac{3n}{4} \rceil$.

Suppose that the case (1) holds, i.e., $n_p \leq \lfloor \frac{n}{4} \rfloor$. Since $N_p(w)+N_p(-w) \geq \frac{3 \dim M}{8}$ and $n_p \leq \lfloor \frac{n}{4} \rfloor$, this means that $p$ must have $+w$ as a weight. Since $w$ is the smallest positive weight, by Lemma \ref{l22}, this implies that there exists a fixed point $p'$ with $n_{p'}=n_p+1$. If $n_{p'} \leq \lfloor \frac{n}{4} \rfloor$, applying the same argument, there exists a fixed point $p''$ with $n_{p''}=n_{p'}+1=n_p+2$. Continuing this, we conclude that there must exist a fixed point $q$ with $\lfloor \frac{n}{4} \rfloor < n_q < \lceil \frac{3n}{4} \rceil$.

Suppose that the case (3) holds, i.e., $\lceil \frac{3n}{4} \rceil \geq n_p$. Since $N_p(w)+N_p(-w) \geq \frac{3 \dim M}{8}$ and $\lceil \frac{3n}{4} \rceil \geq n_p$, this means that $p$ must have $-w$ as a weight. Since $w$ is the smallest positive weight, by Lemma \ref{l22}, this implies that there exists a fixed point $p'$ with $n_{p'}=n_p-1$. If $\lceil \frac{3n}{4} \rceil \geq n_{p'}$, applying the same argument, there exists a fixed point $p''$ with $n_{p''}=n_{p'}-1=n_p-2$.  Continuing this, we conclude that there must exist a fixed point $q$ with $\lfloor \frac{n}{4} \rfloor < n_q <\lceil \frac{3n}{4} \rceil$.

Therefore, in the three cases above, the case (2) holds; we have that there exists a fixed point $q$ such that $\lfloor \frac{n}{4} \rfloor < n_q <\lceil \frac{3n}{4} \rceil$. Since $N_q(w)+N_q(-w) \geq \frac{3 \dim M}{8}$, this implies that $q$ has $+w$ as weight and $-w$ as weight, i.e., $N_q(w)>0$ and $N_q(-w)>0$. Then the following hold.
\begin{enumerate}[(i)]
\item Since $N_q(w)>0$ and $w$ is the smallest positive weight, by Lemma \ref{l22} this implies that there exists a fixed point $q_1$ with $n_{q_1}=n_q+1$. 
\item Since $N_q(-w)>0$ and $w$ is the smallest positive weight, by Lemma \ref{l22} this implies that there exists a fixed point $q^1$ with $n_{q^1}=n_q-1$.
\end{enumerate}

Next, consider $q_1$. If $\lfloor \frac{n}{4} \rfloor < n_{q_1} < \lceil \frac{3n}{4} \rceil$, by (i) above there exists a fixed point $q_2$ such that $n_{q_2}=n_{q_1}+1$. Next, consider $q_2$. If $\lfloor \frac{n}{4} \rfloor < n_{q_2} < \lceil \frac{3n}{4} \rceil$, by (i) above there exists a fixed point $q_3$ such that $n_{q_3}=n_{q_2}+1$. Continuing this, we have that for all $i$ such that $n_q < i < \lceil \frac{3n}{4} \rceil$, there exists a fixed point $p_i$ such that $n_{p_i}=i$. Finally, consider a fixed point $q_{j_1}$ such that $n_{q_{j_1}}=\lceil \frac{3n}{4} \rceil-1$. Since $\lfloor \frac{n}{4} \rfloor < n_{q_{j_1}} < \lceil \frac{3n}{4} \rceil$, by (i) above there exists a fixed point $q_{j_2}$ such that $n_{q_{j_2}}=n_{q_{j_1}}+1$.

Similarly, consider $q^1$. If $\lfloor \frac{n}{4} \rfloor < n_{q^1} <\lceil \frac{3n}{4} \rceil$, by (ii) above there exists a fixed point $q^2$ such that $n_{q^2}=n_{q^1}-1$. Next, consider $q^2$. If $\lfloor \frac{n}{4} \rfloor < n_{q^2} < \lceil \frac{3n}{4} \rceil$, by (ii) above there exists a fixed point $q^3$ such that $n_{q^3}=n_{q^2}-1$. Continuing this, we have that for all $i$ such that $\lfloor \frac{n}{4} \rfloor < i < n_q$, there exists a fixed point $p_i$ such that $n_{p_i}=i$. Finally, consider a fixed point $q_{j_3}$ such that $n_{q_{j_3}}=\lfloor \frac{n}{4} \rfloor+1$. Since $\lfloor \frac{n}{4} \rfloor < n_{q_{j_3}} <\lceil \frac{3n}{4} \rceil$, by (i) above there exists a fixed point $q_{j_4}$ such that $n_{q_{j_4}}=n_{q_{j_3}}-1$.

Therefore, for $i$ such that $\lfloor \frac{n}{4} \rfloor \leq i \leq \lceil \frac{3n}{4} \rceil$, there exists a fixed point $p_i$ such that $n_{p_i}=i$. It follows that there are at least $\lceil \frac{3n}{4} \rceil-\lfloor \frac{n}{4} \rfloor+1$ fixed points. Since $\lfloor \frac{\dim M}{4} \rfloor+1 \leq \lceil \frac{3n}{4} \rceil-\lfloor \frac{n}{4} \rfloor+1$ for any non-negative integer $n$, it follows in particular that $M$ has at least $\lfloor \frac{\dim M}{4} \rfloor+1$ fixed points. \end{proof}

With Lemma \ref{l31}, we are ready to prove our main result, Theorem \ref{t16}.

\begin{proof}[Proof of Theorem \ref{t16}] The theorem holds if there is a fixed component of positive dimension. By Corollary \ref{t13}, the theorem holds if $\dim M \leq 14$. If $a=b$, then the theorem holds by Proposition \ref{t15}. Without loss of generality, by quotienting out by the subgroup which acts trivially, we may assume that the action is effective. Therefore, from now on, suppose that the fixed point set is discrete, $\dim M > 14$, the action is effective, and $a \neq b$. Since the action is effective, $a$ and $b$ are pairwise prime. Without loss of generality, let $a<b$.

Consider the subgroup $\mathbb{Z}_b$ of $S^1$. As a subgroup of $S^1$, $\mathbb{Z}_b$ acts on $M$. The set $M^{\mathbb{Z}_b}$ of points fixed by the $\mathbb{Z}_b$-action is a union of smaller dimensional compact almost complex manifolds. Let $Z$ be a connected component of $M^{\mathbb{Z}_b}$.

On $Z$, there is an induced action of $S^1/\mathbb{Z}_b=S^1$. If $p$ is a fixed point of the induced action on $Z$, then the weights at $p$ in $T_pZ$ of the induced action are $\pm b$. This means that every weight at a point which is fixed by the induced action on $Z$ is $\pm b$. Applying Lemma \ref{l22} to the induced circle action on $Z$, the induced action on $Z$ has $r \cdot 2^m$ fixed points for some positive integer $r$, where $\dim Z=2m$. In particular, the induced action has at least $2^m$ fixed points.

Denote by $Z^{S^1}$ the set of points in $Z$ that are fixed by the induced $S^1$-action. If $p \in Z^{S^1}$, then $p$ is also fixed by the original $S^1$-action on $M$. Conversely, if $p$ is a point in $Z$ which is fixed by the circle action on $M$, then $p$ is also fixed by the induced $S^1$-action on $Z$. This means that $Z^{S^1}=Z \cap M^{S^1}$.

Suppose that $8m > \dim M$. Since $\dim M > 14$ and $8m > \dim M$, we have that $\dim M < 4 \cdot 2^m$. Since $Z^{S^1}$ consists of at least $2^m$ points and $Z^{S^1}=Z \cap M^{S^1}$, $M$ has at least $2^m$ fixed points. Thus, we have that $\dim M < 4 \cdot 2^m \leq 4 \cdot |M^{S^1}|$ and hence the theorem holds.


Therefore, from now on, suppose that $8m \leq \dim M$. This means that $m \leq \frac{\dim M}{8}$. In other words, this means that at each fixed point $p \in M^{S^1}$, $N_p(+b)+N_b(-b) \leq \frac{\dim M}{8}$. Hence, we have that $N_p(a)+N_p(-a) \geq \frac{3 \dim M}{8}$. By Lemma \ref{l31}, there are at least $\lfloor \frac{\dim M}{4} \rfloor+1$ fixed points. \end{proof}

\begin{proof}[Proof of Theorem \ref{t32}] If there is a fixed component of positive dimension, then the theorem holds. Therefore, from now on, suppose that the fixed point set is discrete. The idea of the proof is to consider the set $M^{\mathbb{Z}_{w_i}}$ of points fixed by the $\mathbb{Z}_{w_i}$-action as in the proof of Theorem \ref{t16}.

Choose a fixed point $p \in M^{S^1}$. Fix $i$. Consider the set $M^{\mathbb{Z}_{w_i}}$ of points fixed by the $\mathbb{Z}_{w_i}$-action, where $\mathbb{Z}_{w_i}$ acts on $M$ as a subgroup of $S^1$. Let $Z_i$ be a connected component of $M^{\mathbb{Z}_{w_i}}$ which contains $p$. Consider the induced action $S^1/\mathbb{Z}_{w_i}=S^1$ on $Z_i$. Let $\dim Z_i=2m_i$.

If $q$ is a point which is fixed by the induced $S^1$-action on $Z_i$, every weight of the induced action on $Z_i$ at $q$ in $T_q Z_i$ is equal to $+w_i$ or $-w_i$, since the weights $w_j$ are pairwise relatively prime. As in the proof of Theorem \ref{t16}, the set $Z_{1}^{S^1}$ of points in $Z_{1}$ fixed by the induced circle action is equal to $Z_{1} \cap M^{S^1}$, i.e., $Z_i^{S^1}=Z_i \cap M^{S^1}$ and hence $Z_i^{S^1}$ is a discrete set. Therefore, applying Lemma \ref{l22} to the induced $S^1$-action on $Z_i$, the induced action on $Z_i$ has $r_i \cdot 2^{m_i}$ fixed points for some positive integer $r_i$. In particular, the induced action on $Z_i$ has at least $2^{m_i}$ fixed points ($m_i$ may be equal to 0 and if $m_i=0$, then $q$ does not have weights $w_i$ or $-w_i$).

Since every weight at $p$ is equal to $\pm w_i$ for some $1 \leq i \leq l$, we have that $m_1 + m_2 + \cdots+ m_l=n$. In particular, we have that $m=\max\{m_1,\cdots,m_l\} \geq \frac{n}{l}$. 

Without loss of generality, assume that $m=m_1$. The induced action of $S^1/\mathbb{Z}_{w_{1}}=S^1$ on $Z_{1}$ has at least $2^{m_{1}}$ fixed points. Since $Z_1^{S^1}=Z_1 \cap M^{S^1} \subset M^{S^1}$, this implies that the circle action on $M$ has at least $2^{m_{1}}$ fixed points.

By the assumption, $\displaystyle l < \frac{\dim M}{2[(\log_2 \dim M) -2]}$. This is equivalent to $\displaystyle \dim M < 4 \cdot 2^{\frac{\dim M}{2l}}$. Therefore, we have that $\displaystyle \dim M < 4 \cdot 2^{\frac{\dim M}{2l}} \leq 4 \cdot 2^{m_{1}} \leq 4 \cdot |M^{S^1}|$. Therefore, there exist at least $\lfloor \frac{\dim M}{4} \rfloor+1$ fixed points. \end{proof}

\begin{proof}[Proof of Theorem \ref{t113}] By Corollary \ref{t13}, the theorem holds if $\dim M \leq 14$. Therefore, assume that $\dim M > 14$. Then we have $\displaystyle \dim M < 4 \cdot 2^{\frac{\dim M}{6}}$, which is equivalent to $\displaystyle 3 < \frac{\dim M}{2[(\log_2 \dim M) -2]}$. Therefore, the theorem follows by Theorem \ref{t32}. \end{proof}

\end{document}